\documentclass[a4paper,10pt]{amsart}%
\usepackage{amsfonts, amsmath, amssymb, amsbsy, amstext, amsxtra, latexsym, mathabx}
\usepackage[all]{xy}
\usepackage{graphicx}%
\providecommand{\U}[1]{\protect\rule{.1in}{.1in}}

\newtheorem{thm}{Theorem}[section]

\newtheorem{assum}[thm]{Assumption}
\newtheorem{nota}[thm]{Notation}
\newtheorem{lemma}[thm]{Lemma}
\newtheorem{prop}[thm]{Proposition}
\newtheorem{cor}[thm]{Corollary}

\newtheorem{rmk}[thm]{Remark}

\subjclass[2000]{}

\title[Burniat surfaces with $K^{2}=4$ and of non nodal type]{A characterization of Burniat surfaces with $K^{2}=4$ and of non nodal type}
\author{\uppercase{SHIN Y\lowercase{ong}J\lowercase{oo}}}
\date{}

\address{Shanghai Center for Mathematical Sciences, 22 Floor, East Guanghua Tower, Fudan University, No. 220 Handan Rd., Shanghai 200433, P. R. China}
\email{shinyongjoo@fudan.edu.cn}

\begin{document}

\begin{abstract}
Let $S$ be a minimal surface of general type with $p_{g}(S)=0$ and $K^{2}_{S}=4$.\ Assume the bicanonical map $\varphi$ of $S$ is a morphism of degree $4$ such that the image of $\varphi$ is smooth.\ Then we prove that the surface $S$ is a Burniat surface with $K^{2}=4$ and of non nodal type.
\end{abstract}

\maketitle

\section{\bigskip Introduction}
When we consider the bicanonical map $\varphi$ of a minimal surface $S$ of general type with $p_{g}(S)=0$ over the field of complex numbers, Xiao \cite{FABSTG} gave that the image of $\varphi$ is a surface if $K^{2}_{S}\ge 2$, and Bombieri \cite{CMSGT} and Reider  \cite{VBR2L} proved that $\varphi$ is a morphism for $K^{2}_{S}\ge 5$. In \cite{SG06, BS07, DBS0} Mendes Lopes \cite{DGCSG0} and Pardini obtained that the degree of $\varphi$ is $1$ for $K^{2}_{S}=9$; $1$ or $2$ for $K^{2}_{S}=7,8$; $1,2$ or $4$ for $K^{2}_{S}=5,6$ or for $K^{2}_{S}=3,4$ with a morphism $\varphi$. Moreover, there are further studies for the surface $S$ with non birational map $\varphi$ in \cite{CCMSS0, NS03, ESEN, SG06, BS07II, CDPG80}. 
 
 Mendes Lopes and Pardini \cite{CCMSS0} gave a characterization of a Burniat surface with $K^{2}=6$ as a minimal surface of general type with $p_{g}=0,\ K^{2}=6$ and the bicanonical map of degree $4$. Zhang \cite{CCS05BM} proved that  a surface $S$ is a Burniat surface with $K^{2}=5$  if the image of the bicanonical map $\varphi$ of $S$ is smooth, where a surface $S$ is a minimal surface of general type with $p_{g}(S)=0,\ K^{2}_{S}=5$ and the bicanonical map $\varphi$ of degree $4$. In this paper we extend their characterizations of Burniat surfaces with $K^{2}=6$ \cite{CCMSS0}, and with $K^{2}=5$ \cite{CCS05BM} to one for the case $K^{2}=4$ as the following. 
 
  \begin{thm}\label{mainthm}
Let $S$ be a minimal surface of general type with $p_{g}(S)=0$ and $K_{S}^{2}=4$. Assume the bicanonical map $\varphi\colon S\longrightarrow \Sigma\subset \mathbb{P}^{4}$ is a morphism of degree $4$ such that the image $\Sigma$ of $\varphi$ is smooth. Then the surface $S$ is a Burniat surface with $K^{2}=4$ and of non nodal type.
\end{thm}

 As we mentioned before Bombieri \cite{CMSGT} and Reider \cite{VBR2L} gave that the bicanonical map of a minimal surface of general type with $p_{g}=0$ is a morphism for $K^{2}\ge 5$. On the other hand, Mendes Lopes and Pardini \cite{NCF9} found that there is a family of numerical Campedelli surfaces, minimal surfaces of general type with $p_{g}=0$ and $K^{2}=2$, with $\pi_{1}^{alg}=\mathbb{Z}_{3}^{2}$ such that the base locus of the bicanonical system consists of two points. However, we do not know whether the bicanonical system of a minimal surface of general type with $p_{g}=0$ and $K^{2}=3$ or $4$ has a base point or not. Thus we need to assume that the bicanonical map is a morphism in Theorem \ref{mainthm}.

Bauer and Catanese \cite{BSSBII, BSSBIII, BSSBEII} studied Burniat surfaces with $K^{2}=4$. Let $S$ be a Burniat surface with $K^{2}=4$. When $S$ is of non nodal type it has the ample canonical divisor, but when $S$ is of nodal type it has one ($-2$)-curve. For the case of nodal type we will discuss to characterize Burniat surfaces with $K^{2}=4$ and of nodal type in the future article.

We follow and use the strategies of Mendes Lopes and Pardini \cite{CCMSS0}, and of Zhang \cite{CCS05BM} as main tools of this article. The paper is organized as follows: in Section \ref{DCBC} we recall some useful formulas and Propositions for a double cover from \cite{CCMSS0}, and we give a description of a Burniat surface with $K^{2}=4$ and of non nodal type; in Section \ref{analyze} we analyze branch divisors of the bicanonical morphism $\varphi$ of degree $4$ of a minimal surface of general type with $p_{g}=0$ and $K^{2}=4$ when the image of $\varphi$ is smooth; in Section \ref{proofmainthm} we give a proof of Theorem \ref{mainthm}.

\section{Notation and conventions}\label{NC}
In this section we fix the notation which will be used in the paper. We work over the field of complex numbers.

Let $X$ be a smooth projective surface. Let $\Gamma$ be a curve in
$X$ and $\tilde{\Gamma}$ be the normalization of $\Gamma$. We set:\\\\
$K_X$: the canonical divisor of $X$;\\
$q(X)$: the irregularity of $X$, that is, $h^{1}(X,\mathcal{O}_{X})$;\\
$p_{g}(X)$: the geometric genus of $X$, that is, $h^{0}(X,\mathcal{O}_{X}(K_{X}))$;\\
$p_{g}(\Gamma)$: the geometric genus of $\Gamma$, that is,
$h^{0}(\tilde{\Gamma},\mathcal{O}_{\tilde{\Gamma}}(K_{\tilde{\Gamma}}))$;\\
$\chi_{top}(X)$: the topological Euler characteristic of $X$;\\
$\chi(\mathcal{F})$: the Euler characteristic of a sheaf $\mathcal{F}$ on $X$, that is, $\sum_{i=0}^{2}(-1)^{i}h^{i}(X,\mathcal{F})$;\\
$\equiv$: the linear equivalence of divisors on a surface;\\ 
$(-n)$-curve: a smooth irreducible rational curve with the self-intersection number $-n$,
in particular we call that a $(-1)$-curve is exceptional and a $(-2)$-curve is nodal;\\
We usually omit the sign $\cdot$ of the intersection product of two
divisors on a surface.  And we do not distinguish between line bundles and divisors on a smooth variety.

\section{Preliminaries}\label{DCBC}

\subsection{Double covers}\label{DC}
Let $S$ be a smooth surface and $B\subset S$ be a smooth curve (possibly empty) such that $2L\equiv B$ for a line bundle $L$ on $S$. Then there exists a double cover $\pi\colon Y\longrightarrow S$ branched over $B$. We get \[\pi_{*}\mathcal{O}_{Y}=\mathcal{O}_{S}\oplus L^{-1},\] and the invariants of $Y$ from ones of $S$ as follows:
\[K^{2}_{Y}=2(K_{S}+L)^{2},\ \chi(\mathcal{O}_{Y})=2\chi(\mathcal{O}_{S})+\frac{1}{2}L(K_{S}+L),\]
\[p_{g}(Y)=p_{g}(S)+h^{0}(S,\mathcal{O}_{S}(K_{S}+L)),\]
\[q(Y)=q(S)+h^{1}(S,\mathcal{O}_{S}(K_{S}+L)).\]
  We begin with the following Proposition in \cite{CCMSS0}. 
\begin{prop}[Proposition 2.1 in \cite{CCMSS0}]\label{albenese}
Let $S$ be a smooth surface with $p_{g}(S)=q(S)=0$, and let $\pi\colon Y\longrightarrow S$ be a smooth double cover. Suppose that $q(Y)>0$. Denote the Albanese map of $Y$ by $\alpha\colon Y\longrightarrow$ A. Then

$(i)$ the Albanese image of $Y$ is a curve $C$$;$

$(ii)$ there exist a fibration $g\colon S \longrightarrow \mathbb{P}^{1}$ and a degree $2$ map $p\colon C\longrightarrow \mathbb{P}^{1}$ such that $p\circ\alpha=g\circ\pi$.
\end{prop}

\begin{prop}[Corollary 2.2 in \cite{CCMSS0}]\label{inequiKq}
Let $S$ be a smooth surface of general type with $p_{g}(S)=q(S)=0,\ K_{S}^{2}\ge3$, and let $\pi\colon Y\longrightarrow S$ be a smooth double cover. Then $K_{Y}^{2}\ge16(q(Y)-1)$.
\end{prop}

\subsection{Bidouble covers} \label{bico} Let $Y$ be a smooth surface and $D_{i}\subset Y,\ i=1,2,3$ be smooth divisors such that $D:=D_{1}+D_{2}+D_{3}$ is a normal crossing divisor, $2L_{1}\equiv D_{2}+D_{3}$ and $2L_{2}\equiv D_{1}+D_{3}$ for line bundles $L_{1},\ L_{2}$ on $Y$. By \cite{ACAV} there exists a bidouble cover $\psi\colon \bar{Y}\longrightarrow Y$ branched over $D$. We obtain \[\psi_{*}\mathcal{O}_{\bar{Y}}=\mathcal{O}_{Y}\oplus L_{1}^{-1}\oplus L_{2}^{-1}\oplus L_{3}^{-1},\] where $L_{3}=L_{1}+L_{2}-D_{3}$.
 
 We describe a Burniat surface with $K^{2}=4$ and of non nodal type \cite{BSSBII}. 
 \begin{nota} \label{notdel}
{\rm{ Let $\rho\colon \Sigma\longrightarrow \mathbb{P}^{2}$ be the blow-up of $\mathbb{P}^{2}$ at $5$ points $p_{1},\ p_{2},\ p_{3},\ p_{4},\ p_{5}$ in general position. We denote that $l$ is the pull-back of a line in $\mathbb{P}^{2}$, $e_{i}$ is  the exceptional curve over $p_{i},\ i=1,2,3,4,5$, and $e'_{i}$ is the strict transform of the line joining $p_{j}$ and $p_{k},\ \{i,j,k\}=\{1,2,3\}$. Also, $g_{i}$ $(resp.\ h_{i})$ denotes the strict transform of the line joining $p_{4}$ $(resp.\ p_{5})$ and $p_{i},\ i=1,2,3$. Then the picard group of $\Sigma$ is generated by $l,\ e_{1},\ e_{2},\ e_{3},\ e_{4}$ and $e_{5}$. We get that $-K_{\Sigma}\equiv3l-\sum_{i=1}^{5}e_{i}$ is very ample. The surface $\Sigma$ is embedded by the linear system $|-K_{\Sigma}|$ as a smooth surface of degree $4$ in $\mathbb{P}^{4}$, called a del Pezzo surface of degree $4$.
}}
 \end{nota}
 
 We consider smooth divisors
  \[B_{1}=e_{1}+e'_{1}+g_{2}+h_{2}\equiv3l+e_{1}-3e_{2}-e_{3}-e_{4}-e_{5}, \textrm{ }\textrm{ }\textrm{ }\textrm{ }\textrm{ }\]
  \[B_{2}=e_{2}+e'_{2}+g_{3}+h_{3}\equiv3l-e_{1}+e_{2}-3e_{3}-e_{4}-e_{5}, \textrm{and}\]
  \[B_{3}=e_{3}+e'_{3}+g_{1}+h_{1}\equiv3l-3e_{1}-e_{2}+e_{3}-e_{4}-e_{5}.\textrm{ }\textrm{ }\textrm{ }\textrm{ }\textrm{ }\]
 Then $B:=B_{1}+B_{2}+B_{3}$ is a normal crossing divisor, $2L'_{1}\equiv B_{2}+B_{3}$ and $2L'_{2}\equiv B_{1}+B_{3}$ for line bundles $L'_{1},\ L'_{2}$ on $\Sigma$. We obtain a bidouble cover $\varphi\colon S\longrightarrow \Sigma\subset \mathbb{P}^{4}$. We remark that the example is a minimal surface $S$ of general type with $p_{g}(S)=0,\ K_{S}^{2}=4$ and the bicanonical morphism $\varphi$ of degree $4$ having the ample $K_{S}$.

\section{Branch divisors of the bicanonical map}\label{analyze}
\begin{nota}\label{BDBM1}
{\rm{
Let $S$ be a minimal surface of general type with $p_{g}(S)=0$ and $K_{S}^{2}=4$. Assume that the bicanonical map $\varphi$ of $S$ is a morphism of degree $4$ and the image $\Sigma$ of $\varphi$ is smooth in $\mathbb{P}^{4}$. By \cite{ORS} $\Sigma$ is a del Pezzo surface of degree $4$ in Notation \ref{notdel}. We denote $\rho,l,e_{i},e'_{j},g_{j},h_{j},\ i=1,2,3,4,5,\ j=1,2,3$ as the notations in Notation \ref{notdel}. Denote $\gamma\equiv l-e_{4}-e_{5},\ \delta\equiv 2l-\sum_{i=1}^{5}e_{i},\ f_{i}\equiv l-e_{i}$ and $F_{i}\equiv\varphi^{*}(f_{i})$ for $i=1,2,3,4,5$.
}}
\end{nota}
We follow the strategies of \cite{CCMSS0, CCS05BM}. We start with the following proposition similar to one in Section $4$ of \cite{CCS05BM}.

\begin{prop}[Note Proposition 4.2 in \cite{CCS05BM}]\label{cong3}
For $i=1,2,3,4,5$ if $f_{i}\in |f_{i}|$ is general, then $\varphi^{*}(f_{i})$ is connected, hence $|F_{i}|$ induces a genus $3$ fibration $u_{i}\colon S\longrightarrow \mathbb{P}^{1}$.
\end{prop}
\begin{proof}
We get a similar proof from Proposition 4.2 in \cite{CCS05BM}.
\end{proof}

\begin{prop}[Note Proposition 4.4 in \cite{CCMSS0} and Proposition 4.3 in \cite{CCS05BM}]\label{finampirr}
The bicanonical morphism $\varphi$ is finite, the canonical divisor $K_{S}$ is ample, and for $i=1,2,3,4,5$, the pull-back of an irreducible curve in $|f_{i}|$ is also irreducible (possibly non-reduced).
\end{prop}

\begin{proof}
Noether's formula gives $\chi_{top}(S)=8$ by $\chi(\mathcal{O}_{S})=1$ and $K_{S}^{2}=4$. Then we get $h^{2}(S,\mathbb{Q})=h^{2}(\Sigma,\mathbb{Q})=6$ by $p_{g}(S)=q(S)=0$. So $\varphi^{*}\colon H^{2}(\Sigma,\mathbb{Q})\longrightarrow H^{2}(S,\mathbb{Q})$ is an isomorphism preserving the intersection form up to multiplication by $4$.\ Therefore $\varphi$ is finite and $K_{S}$ is ample.

For an irreducible curve $f_{1}\in|f_{1}|$, if $\varphi^{*}(f_{1})$ is reducible, then it contains an irreducible component $C$ with $C^{2}<0$. Put $D=C-\frac{C\varphi^{*}(e_{1})}{4}\varphi^{*}(f_{1})$.\ Then $D^{2}=C^{2}<0$, and $D\varphi^{*}(e_{1})=0$. And $\left(C-\frac{C\varphi^{*}(e_{1})}{4}\varphi^{*}(f_{1})\right)\varphi^{*}(e_{i})=0$ for $i=2,3,4,5$ since $e_{i}$ is contained in one fiber of the pencil $|f_{1}|$.\ We obtain that the intersection matrix of $\varphi^{*}(l),\ C-\frac{C\varphi^{*}(e_{1})}{4}\varphi^{*}(f_{1}),\ \varphi^{*}(e_{i}),\ i=1,2,3,4,5$ has rank $7$.\ But it is a contradiction because $h^{2}(S,\mathbb{Q})=6$. Thus $\varphi^{*}(f_{1})$ is irreducible. We can similarly prove the other cases.
\end{proof}

\begin{lemma}[Lemma 4.4 in \cite{CCS05BM}]\label{Mm}
Let $\phi\colon T'\longrightarrow T$ be a finite morphism between two smooth surfaces. Let $h$ be a divisor on $T$ such that $|\phi^{*}(h)|=\phi^{*}(|h|)$. Let $M$ be a divisor on $T'$ such that the linear system $|M|$ has no fixed part. Suppose that $\phi^{*}(h)-M$ is effective. Then there exists a divisor $m\subset T$ such that $|M|=\phi^{*}(|m|)$. Furthermore the line bundle $h-m$ is effective.
\end{lemma}

\begin{lemma}[Note Lemma 4.5 in \cite{CCS05BM}]\label{ne1e}
There does not exist a divisor $d$ on $\Sigma$ such that $h^{0}(\Sigma,\mathcal{O}_{\Sigma}(d))>1$ and that the line bundle $-K_{\Sigma}-2d$ is effective.
\end{lemma}
\begin{proof}
Suppose that there exists such a divisor $d$. Assume $d\equiv al-\sum_{i=1}^{5}b_{i}e_{i}$ for some integers $a,\ b_{i},\ i=1,2,3,4,5$. Then $a\le1$ because $-K_{\Sigma}-2d$ is effective. On the other hand, $a\ge 1$ by the condition $h^{0}(\Sigma,\mathcal{O}_{\Sigma}(d))>1$. Thus $a=1$, and at most one of $b_{1},\cdots,b_{5}$ is positive. Then the line bundle $-K_{\Sigma}-2d\equiv l-\sum_{i=1}^{5}(1-2b_{i})e_{i}$ cannot be effective since there is no line on $\mathbb{P}^{2}$ passing through $3$ points in general position.
\end{proof}

We prove the following lemma as one of Lemma 4.6 in \cite{CCS05BM} since we have Lemmas \ref{Mm} and \ref{ne1e}. 

\begin{lemma}[Note Lemma 4.6 in \cite{CCS05BM}]\label{2Deff1}
Let $D\subset S$ be a divisor. If there exists a divisor $d$ on $\Sigma$ such that

$(i)$ $\varphi^{*}(d)\equiv 2D;$

$(ii)$ the line bundle $-K_{\Sigma}-d$  is effective,\\
then $h^{0}(S,\mathcal{O}_{S}(D))\le 1$.
\end{lemma}
\begin{proof}
Suppose $h^{0}(S,\mathcal{O}_{S}(D))>1$. We may write $|D|=|M|+F$ where $|M|$ is the moving part and $F$ is the fixed part. Since $|2K_{S}|=|\varphi^{*}(-K_{\Sigma})|=\varphi^{*}(|-K_{\Sigma}|)$ and $\varphi^{*}(-K_{\Sigma})-M>\varphi^{*}(-K_{\Sigma}-d)$ is effective, there is a divisor $m$ on $\Sigma$ such that $\varphi^{*}(|m|)=|M|$ by Lemma \ref{Mm}. Choose an element $M_{1}\in |M|$ and an effective divisor $N$ on $S$ such that $2M_{1}+2F+N\equiv \varphi^{*}(-K_{\Sigma})$. We find $h\in|-K_{\Sigma}|$ and $m_{1}\in|m|$ such that $2M_{1}+2F+N=\varphi^{*}(h)$ and $2M_{1}=\varphi^{*}(2m_{1})$. Thus we conclude that $h-2m_{1}$ is effective. It is a contradiction by Lemma \ref{ne1e}.
\end{proof}

Now we investigate the pull-backs of ($-1$)-curves on the surface $\Sigma$ via the bicanonical morphism $\varphi\colon S\longrightarrow \Sigma \subset \mathbb{P}^{4}$. There are sixteen $(-1)$-curves on $\Sigma$ which are $e_{i},\ e'_{j},\ g_{j},\ h_{j},$ $\gamma$ and $\delta$ for $i=1,2,3,4,5$ and $j=1,2,3$.

\begin{lemma}[Lemma 5.1 in \cite{CCMSS0}]\label{twocases}
Let $C\subset\Sigma$ be a $(-1)$-curve. Then we have either

$(i)$ $\varphi^{*}(C)$ is a reduced smooth rational $(-4)$-curve$;$ or

$(ii)$ $\varphi^{*}(C)=2E$ where $E$ is an irreducible curve with $E^{2}=-1,\ K_{S}E=1$.
\end{lemma}

\begin{lemma}\label{three-4}
There are at most three disjoint $(-4)$-curves on $S$.
\end{lemma}
\begin{proof}
Let $r$ be the cardinality of a set of smooth disjoint rational curves with self-intersection number $-4$ on $S$. Then 
\[\frac{25}{12}r\le c_{2}(S)-\frac{1}{3}K_{S}^{2}=\frac{20}{3}\]
by \cite{mnqs}
which is $r\le 3$.
\end{proof}

\begin{rmk}\label{cremona}
\rm{We consider an exceptional curve $e$ on $\Sigma$ which is different from $\delta$ and is not an $\rho$-exceptional curve (i.e.\ $\rho(e)$ is not a point in $\mathbb{P}^{2}$). Then we can find an automorphism $\tau$ on $\Sigma$ induced by a Cremona transformation with respect to 3 points among 5 points $p_{1},p_{2},p_{3},p_{4},p_{5}$ in general position on $\mathbb{P}^{2}$ such that an exceptional curve $\tau(e)$ on $\Sigma$ is different from $\delta$ and is an $\rho$-exceptional curve.}
\end{rmk}

\begin{prop}\label{theretwo}
There exist at least two disjoint $(-1)$-curves different from $\delta$ on $\Sigma$ such that those pull-backs by the bicanonical morphism $\varphi$ are $(-4)$-curves.
\end{prop}
\begin{proof}
Let $R$ be the ramification divisor of the bicanonical morphism $\varphi\colon S\longrightarrow \Sigma \subset \mathbb{P}^{4}$. By Hurwitz formula $K_{S}\equiv \varphi^{*}(K_{\Sigma})+R$, we get $R\equiv  K_{S}+\varphi^{*}(-K_{\Sigma})\equiv3K_{S}$. Because $\varphi^{*}(-K_{\Sigma})\equiv2K_{S}$ since the image $\Sigma$ of $\varphi$ is a del Pezzo surface of degree $4$ in $\mathbb{P}^{4}$ (Note Notations \ref{notdel} and \ref{BDBM1}). 

We assume $\varphi^{*}(e_{i})=2E_{i}$, $\varphi^{*}(e'_{j})=2E'_{j}$, $\varphi^{*}(g_{j})=2G_{j}$, $\varphi^{*}(h_{j})=2H_{j}$ for $i=1,2,3,4,5$ and $j=1,2,3$, and $\varphi^{*}(\gamma)=2\Gamma$. Put \[R_{1}:= R-\left(\sum_{i=1}^{3}(E_{i}+E'_{i}+G_{i}+H_{i})+E_{4}+E_{5}+\Gamma\right).\]
It implies $2R_{1}\equiv \varphi^{*}(-l)$. By the assumption, $\varphi$ is ramified along reduced curves $E_{i},\ E'_{j},\ G_{j},\ H_{j}$ for $i=1,2,3,4,5$ and $j=1,2,3$, and $\Gamma$. So $R_{1}$ is a nonzero effective divisor. But it is a contradiction because $0< (2R_{1})(2K_{S})=\varphi^{*}(-l)\varphi^{*}(-K_{\Sigma})<0$ since $\varphi$ is finite and $K_{S}$ is ample by Proposition \ref{finampirr}.\ Thus by Lemma \ref{twocases} and Remark \ref{cremona} we may consider $\varphi^{*}(e_{5})=E_{5}$, where $E_{5}$ is a reduced smooth rational $(-4)$-curve. 

Again, we assume $\varphi^{*}(e_{i})=2E_{i}$, $\varphi^{*}(e'_{j})=2E'_{j}$, $\varphi^{*}(g_{j})=2G_{j}$, $\varphi^{*}(h_{j})=2H_{j}$ for $i=1,2,3,4$ and $j=1,2,3$, and $\varphi^{*}(\gamma)=2\Gamma$. Put \[R_{2}:=R-\left(\sum_{i=1}^{3}(E_{i}+E'_{i}+G_{i}+H_{i})+E_{4}+\Gamma\right).\]
It induces $2R_{2}\equiv \varphi^{*}(-l+e_{5})$. Then the nonzero divisor $R_{2}$ is effective. Because $\varphi$ is ramified along reduced curves $E_{i},\ E'_{j},\ G_{j},\ H_{j}$ for $i=1,2,3,4$ and $j=1,2,3$, and $\Gamma$ from the assumption. It gives a contradiction because $0< (2R_{2})(2K_{S})=\varphi^{*}(-l+e_{5})\varphi^{*}(-K_{\Sigma})<0$ since $\varphi$ is finite and $K_{S}$ is ample by Proposition \ref{finampirr}. By Lemma \ref{twocases} we get an $(-1)$-curve $e$ with an $(-4)$-curve $\varphi^{*}(e)$ among $e_{i}$, $e'_{j},\ g_{j},\ h_{j}$ for $i=1,2,3,4$ and $j=1,2,3$, and $\gamma$.
 
We have two $(-1)$-curves $e$ and $e_{5}$ different from $\delta$ on $\Sigma$ such that $\varphi^{*}(e)$ and $\varphi^{*}(e_{5})$ are $(-4)$-curves on $S$. We verify that $e$ and $e_{5}$ are disjoint. By Remark \ref{cremona} we consider that the $(-1)$-curve $e$ is $\gamma$. It is enough to assume $\varphi^{*}(e_{i})=2E_{i}$, $\varphi^{*}(e'_{j})=2E'_{j}$, $\varphi^{*}(g_{j})=2G_{j}$, $\varphi^{*}(h_{j})=2H_{j}$ for $i=1,2,3,4$ and $j=1,2,3$. Then put 
\[R_{3}:=R-\left(\sum_{i=1}^{3}(E_{i}+E'_{i}+G_{i}+H_{i})+E_{4}\right).\]
We get $2R_{3}\equiv \varphi^{*}(-e_{4})$.\ The nonzero divisor $R_{3}$ is effective.\ Because $\varphi$ is ramified along reduced curves $E_{i},\ E'_{j},\ G_{j},\ H_{j}$ for $i=1,2,3,4$ and $j=1,2,3$ from the assumption. It contradicts because $0< (2R_{3})(2K_{S})=\varphi^{*}(-e_{4})\varphi^{*}(-K_{\Sigma})<0$ since $\varphi$ is finite and $K_{S}$ is ample by Proposition \ref{finampirr}.
\end{proof}

\begin{prop}\label{c1c2c3}
There do not exist three $(-1)$-curves $C_{1},C_{2}$ and $C_{3}$ different from $\delta$ on $\Sigma$ satisfying

$(i)$ $C_{i}\cap C_{j}=\varnothing$ for distinct $i,j\in\{1,2,3\};$

$(ii)$ $\varphi^{*}(C_{i})$ for $i=1,2,3$ are $(-4)$-curves.
\end{prop}
\begin{proof}
Assume that the proposition is not true.\ We may consider $C_{1}=e_{2},\ C_{2}=e_{4}$ and $C_{3}=e_{5}$ by Remark \ref{cremona}. Then $E_{2}=\varphi^{*}(e_{2})$, $E_{4}=\varphi^{*}(e_{4})$ and $E_{5}=\varphi^{*}(e_{5})$ are reduced smooth rational $(-4)$-curves. And $\varphi^{*}(e'_{2})=2E'_{2}$ with ${E'_{2}}^{2}=-1,\ K_{S}E'_{2}=1$ by Lemmas \ref{twocases} and \ref{three-4}. Then 

\begin{align*}
2K_{S} &\equiv\varphi^{*}\left(3l-\sum_{i=1}^{5}e_{i}\right)  \equiv \varphi^{*}(e'_{2}+2f_{2}+e_{2}-e_{4}-e_{5})\\
            &\equiv 2E'_{2}+2F_{2}+E_{2}-E_{4}-E_{5}.
\end{align*}
We get $2(K_{S}-E'_{2}-F_{2}+E_{4}+E_{5})\equiv E_{2}+E_{4}+E_{5}$. We consider a double cover $\pi\colon Y\longrightarrow S$ branched over $E_{2},\ E_{4}$ and $E_{5}$. By the formula in Subsection \ref{DC} we obtain
\[K_{Y}^{2}=2(2K_{S}-E'_{2}-F_{2}+E_{4}+E_{5})^{2}=14,\]
\[\chi(\mathcal{O}_{Y})=2+\frac{(K_{S}-E'_{2}-F_{2}+E_{4}+E_{5})\cdot(2K_{S}-E'_{2}-F_{2}+E_{4}+E_{5})}{2}=2,\]
\begin{align*}
p_{g}(Y)  &  =h^{0}(S,\mathcal{O}_{S}(2K_{S}-E'_{2}-F_{2}+E_{4}+E_{5}))\\
               &  =h^{0}(S,\mathcal{O}_{S}(\varphi^{*}(-K_{\Sigma}-e'_{2}-f_{2}+e_{4}+e_{5})+E'_{2}))\\
               &  =h^{0}(S,\mathcal{O}_{S}(\varphi^{*}(l)+E'_{2}))\ge3.
\end{align*}
Thus we have $q(Y)\ge 2$, and so $K_{Y}^{2}<16(q(Y)-1)$.\ It is a contradiction by Proposition \ref{inequiKq}.
\end{proof}

\begin{assum}\label{ass}
\em{From Lemma \ref{twocases}, Propositions \ref{theretwo} and \ref{c1c2c3} we may assume that $\varphi^{*}(e_{4})=E_{4}$ and $\varphi^{*}(e_{5})=E_{5}$ by Remark \ref{cremona}, where $E_{4}$ and $E_{5}$ are $(-4)$-curves, $\varphi^{*}(e_{i})=2E_{i},\ \varphi^{*}(e'_{i})=2E'_{i},\ \varphi^{*}(g_{j})=2G_{j}$ and $\varphi^{*}(h_{j})=2H_{j}$ for $i=1,2,3$ and $j=1,2$.} 
\end{assum}
\begin{nota}\label{etE}
{\rm{$2(E_{j}+E'_{k})$ and $2(E'_{j}+E_{k})$ are two double fibers of $u_{i}\colon S\longrightarrow \mathbb{P}^{1}$ induced by $|F_{i}|$ where $\{i,j,k\}=\{1,2,3\}$. Set $\eta_{i}\equiv (E_{j}+E'_{k})-(E'_{j}+E_{k})$ where $\{i,j,k\}=\{1,2,3\}$, and set $\eta\equiv K_{S}-\sum_{i=1}^{3}(E_{i}+E'_{i})$. Then $2\eta\equiv -E_{4}-E_{5}$, and by Lemma 8.3, Chap. III in \cite{CCS} $\eta_{i}\notequiv0$ for $i=1,2,3$. It implies that $\eta_{i},\ i=1,2,3$ are torsions of order $2$.
}}
\end{nota}

\begin{prop}[Note Proposition 5.9 ($resp.\ 4.13$) in \cite{CCMSS0} ($resp.\ $\cite{CCS05BM})]\label{FijKi}
For a general curve $F_{i}\in |F_{i}|,\ i=1,2,3,$ \[F_{j}|_{F_{i}}\equiv K_{F_{i}}\ \textrm{if}\ i\neq j.\]
\end{prop}
\begin{proof}
We verify that $F_{2}|_{F_{1}}\equiv K_{F_{1}}$. Since $2K_{S}\equiv F_{1}+2(2E_{1}+E'_{3}+E'_{2})-E_{4}-E_{5}$, we get \[2(K_{S}-(2E_{1}+E'_{3}+E'_{2})+E_{4}+E_{5})\equiv F_{1}+E_{4}+E_{5}.\]
It gives a double cover $\pi\colon Y\longrightarrow S$ branched over $F_{1},\ E_{4}$ and $E_{5}$. We have 
\begin{align*}
\chi(\mathcal{O}_{Y})=3
\end{align*}
and
\begin{align*}
p_{g}(Y) & =h^{0}(S,\mathcal{O}_{S}(F_{1}+2E_{1}+E'_{3}+E'_{2}))\\
              & =h^{0}(S, \mathcal{O}_{S}(\varphi^{*}(f_{1}+e_{1})+E'_{3}+E'_{2}))\\
              & =h^{0}(S, \mathcal{O}_{S}(\varphi^{*}(l)+E'_{3}+E'_{2}))\ge 3,
\end{align*}
thus $q(Y)\ge 1$. By Proposition \ref{albenese} the Albanese pencil of $Y$ is the pull-back of a pencil $|F|$ of $S$ such that $\pi^{*}(F)$ is disconnected for a general element $F$ in $|F|$. Thus $FF_{1}=0$ because $\pi$ is branched over $F_{1}$. It means $|F|=|F_{1}|$. For a general element $F_{1}\in |F_{1}|$, $\pi^{*}(F_{1})$ is an unramified double cover of $F_{1}$ given by the relation $2(K_{S}-(2E_{1}+E'_{3}+E'_{2})+E_{4}+E_{5})|_{F_{1}}$. Since $\pi^{*}(F_{1})$ is disconnected, we get
\begin{align*}
(K_{S}-(2E_{1}+E'_{3}+E'_{2})+E_{4}+E_{5})|_{F_{1}} & \equiv (K_{S}-2E_{1})|_{F_{1}}\\ 
                                                                                      & \equiv (K_{S}-2E_{1}-2E'_{3})|_{F_{1}}\\
                                                                                      & \equiv (K_{S}-F_{2})|_{F_{1}}
\end{align*}
is trivial. Thus $F_{2}|_{F_{1}}\equiv K_{F_{1}}$.
\end{proof}

\begin{lemma}\label{invariants} We have:\\
$(i)$ $\chi(\mathcal{O}_{S}(K_{S}+\eta+\eta_{i}))=-1,\ h^{2}(S,\mathcal{O}_{S}(K_{S}+\eta+\eta_{i}))=0;$\\
$(ii)$ $h^{0}(F_{i},\mathcal{O}_{F_{i}}(K_{F_{i}}+\eta|_{F_{i}}))\le 2;$\\
$(iii)$ $h^{1}(S,\mathcal{O}_{S}(\eta-\eta_{i}))=1.$
\end{lemma}
\begin{proof}
$(i)$ By Riemann-Roch theorem, $\chi(S,\mathcal{O}_{S}(K_{S}+\eta+\eta_{i}))=-1$ since $2\eta\equiv -E_{4}-E_{5}$. Moreover, $h^{0}(S,\mathcal{O}_{S}(-\eta+\eta_{i}))=0$ because $2(-\eta+\eta_{i})\equiv E_{4}+E_{5}$ and $E_{4},\ E_{5}$ are reduced $(-4)$-curves. It implies $h^{2}(S,\mathcal{O}_{S}(K_{S}+\eta+\eta_{i}))=0$ by Serre duality.

$(ii)$ We may assume $i=1$. By $\eta_{1}|_{F_{1}}\equiv \mathcal{O}_{F_{1}}$ we have an exact sequence
\[0\longrightarrow \mathcal{O}_{S}(K_{S}+\eta+\eta_{1})\longrightarrow \mathcal{O}_{S}(K_{S}+\eta+\eta_{1}+F_{1})\longrightarrow \mathcal{O}_{F_{1}}(K_{F_{1}}+\eta|_{F_{1}})\longrightarrow 0.\]
Then we get 
\begin{align*}
h^{0}(F_{1},\mathcal{O}_{F_{1}}(K_{F_{1}}+\eta|_{F_{1}})) \le &\ h^{0}(S,K_{S}+\eta+\eta_{1}+F_{1})-h^{0}(S,K_{S}+\eta+\eta_{1})\\
                                                                                             & +h^{1}(S,K_{S}+\eta+\eta_{1})\\
                                                                                            = &\ h^{0}(S,K_{S}+\eta+\eta_{1}+F_{1})-\chi(\mathcal{O}_{S}(K_{S}+\eta+\eta_{1}))\\
                                                                                             & +h^{2}(S,K_{S}+\eta+\eta_{1})\\
                                                                                            = &\ h^{0}(S,\mathcal{O}_{S}(K_{S}+\eta+\eta_{1}+F_{1}))+1. 
\end{align*}
Note $K_{S}+\eta+\eta_{1}+F_{1}\equiv 2K_{S}-(E_{1}+E'_{1})$. Since the linear system $|2K_{S}|$ embeds $E_{1}+E'_{1}$ as a pair of skew lines in $\mathbb{P}^{4}$, we have $h^{0}(S,\mathcal{O}_{S}(2K_{S}-(E_{1}+E'_{1})))=1$. Hence $h^{0}(F_{1},\mathcal{O}_{F_{1}}(K_{F_{1}}+\eta))\le 2$.

$(iii)$ We have $2(\eta-\eta_{i})\equiv -E_{4}-E_{5}$. It implies $h^{0}(S,\mathcal{O}_{S}(\eta-\eta_{i}))=0$. Thus $-h^{1}(S,\mathcal{O}_{S}(\eta-\eta_{i}))+h^{2}(S,\mathcal{O}_{S}(\eta-\eta_{i}))=1$ by Riemann-Roch theorem. We show $h^{0}(S,\mathcal{O}_{S}(K_{S}-\eta+\eta_{1}))=2$ by Serre duality. Indeed, since $E_{4},\ E_{5}$ are rational $(-4)$-curves and $(2K_{S}+E_{4}+E_{5})(E_{4}+E_{5})=0$, we obtain an exact sequence
\[ 0\longrightarrow \mathcal{O}_{S}(2K_{S})  \longrightarrow  \mathcal{O}_{S}(2K_{S}+E_{4}+E_{5})\longrightarrow  \mathcal{O}_{E_{4}\cup E_{5}}  \longrightarrow 0.\]
The canonical divisor $K_{S}$ is ample in Proposition \ref{finampirr}. It follows $h^{0}(S,\mathcal{O}_{S}(2K_{S}))=5$ and $h^{1}(S,\mathcal{O}_{S}(2K_{S}))=0$ by Kodaira vanishing theorem and Riemann-Roch theorem. Thus the long cohomology sequence induces $h^{0}(S,\mathcal{O}_{S}(2K_{S}+E_{4}+E_{5}))=7$. Moreover, since $h^{0}(\Sigma,\mathcal{O}_{\Sigma}(3l-e_{1}-e_{2}-e_{3}))=7$ and $2(K_{S}-\eta+\eta_{1})\equiv 2K_{S}+E_{4}+E_{5}\equiv  \varphi^{*}(3l-e_{1}-e_{2}-e_{3})$, we get $|2(K_{S}-\eta+\eta_{1})|=\varphi^{*}(|3l-e_{1}-e_{2}-e_{3}|)$. Also, $h^{0}(S,\mathcal{O}_{S}(K_{S}-\eta+\eta_{1}))=h^{0}(S,\mathcal{O}_{S}(F_{1}+E'_{1}+E_{1}))\ge2$ because $K_{S}-\eta+\eta_{1}\equiv F_{1}+E'_{1}+E_{1}$.\ We consider $|K_{S}-\eta+\eta_{1}|=|M|+F$ where $|M|$ is the moving part and $F$ is the fixed part.\ By Lemma \ref{Mm} there is a divisor $m$ on $\Sigma$ such that $|M|=\varphi^{*}(|m|)$. Then $3l-e_{1}-e_{2}-e_{3}-2m$ is effective by arguing as in the proof of Lemma \ref{2Deff1}. So $m\equiv f_{i}$ for some $i\in\{1, 2, 3\}$. Hence $h^{0}(S,\mathcal{O}_{S}(K_{S}-\eta+\eta_{1}))=h^{0}(S,\mathcal{O}_{S}(M))=h^{0}(\Sigma, \mathcal{O}_{\Sigma}(f_{i}))=2$. 
\end{proof}

\begin{cor}[Corollary 4.15 in \cite{CCS05BM}]\label{etaij}
For a general curve $F_{i}\in |F_{i}|,\ i=1,2,3$ we have 
\[(-\eta+\eta_{j})|_{F_{i}}\equiv \mathcal{O}_{F_{i}}\ \textrm{if}\ i\neq j;\]
\[ \eta_{i}|_{F_{i}}\equiv \mathcal{O}_{F_{i}};\ (-\eta+\eta_{i})|_{F_{i}}\notequiv\mathcal{O}_{F_{i}}.\]
\end{cor}
\begin{proof}
 By Lemma \ref{FijKi} 
 \begin{align*}
 \eta|_{F_{1}}\equiv (K_{S}-(E_{1}+E'_{1}))|_{F_{1}} & \equiv K_{F_{1}}-(E_{1}+E'_{1})|_{F_{1}}\equiv (F_{2}-(E_{1}+E'_{1}))|_{F_{1}}\\
                                                                                 & \equiv (2(E_{1}+E'_{3})-(E_{1}+E'_{1}))|_{F_{1}}\equiv (E_{1}-E'_{1})|_{F_{1}}.
 \end{align*}
 Since $\eta_{2}|_{F_{1}}\equiv \eta_{3}|_{F_{1}}\equiv(E_{1}-E'_{1})|_{F_{1}}$ we get $(-\eta+\eta_{j})|_{F_{i}}\equiv\mathcal{O}_{F_{i}}$ for $i\neq j$. The definitions of $\eta_{i}$ and $F_{i}$ imply $\eta_{i}|_{F_{i}}\equiv \mathcal{O}_{F_{i}}$. Moreover, if we assume $\eta|_{F_{i}}\equiv \mathcal{O}_{F_{i}}$ then $h^{0}(F_{i},\mathcal{O}_{F_{i}}(K_{F_{i}}+\eta|_{F_{i}}))=h^{0}(F_{i},\mathcal{O}_{F_{i}}(K_{F_{i}}))=3$ because the curve $F_{i}$ has genus $3$ by Proposition \ref{cong3}.\ It induces a contradiction by Lemma \ref{invariants} $(ii)$.
\end{proof}

\section{Proof of Theorem \ref{mainthm}}\label{proofmainthm}

We provide the characterization of Burniat surfaces with $K^{2}=4$ and of non nodal type.\ We use the notations in Notations \ref{BDBM1} and \ref{etE}, and we work with Assumption \ref{ass}. We follow the approaches in \cite{CCMSS0, CCS05BM}.
\begin{lemma}[Note Lemma 5.1 in \cite{CCS05BM}]\label{threefibers}
Let $u\colon S\longrightarrow \mathbb{P}^{1}$ be a fibration such that $E_{4}$ and $E_{5}$ are contained in fibers.\ Then $u$ is induced by one of the pencils $|F_{i}|,\ i=1,2,3$.
\end{lemma}
\begin{proof}
We argue as in the proof of Lemma 5.7 in \cite{CCMSS0}.
\end{proof}

\begin{rmk}
\em{In Lemma \ref{threefibers} $E_{4}$ and $E_{5}$ are not contained in the same fiber of $u$ because $u$ is induced by one of $|F_{i}|,\ i=1,2,3$.}
\end{rmk}

\textit{Proof of Theorem \ref{mainthm}.}
Let $\pi_{i}\colon Y_{i}\longrightarrow S$ be the double cover branched over $E_{4}$ and $E_{5}$ given by the relation $2(-\eta+\eta_{i})\equiv E_{4}+E_{5}$. By Corollary \ref{etaij} $\eta_{i}\notequiv \eta_{j}$ for $i\neq j$. So $\pi_{i}$ is different from $\pi_{j}$. Serre duality and the formula for $q(Y)$ in Subsection \ref{DC} imply $q(Y_{i})=h^{1}(S,\mathcal{O}_{S}(\eta-\eta_{i}))=1$ from Lemma \ref{invariants} $(iii)$. Let $\alpha_{i}\colon Y_{i}\longrightarrow C_{i}$ be the Albanese pencil where $C_{i}$ is an elliptic curve. By Proposition \ref{albenese} there exist a fibration $h_{i}\colon S\longrightarrow \mathbb{P}^{1}$ and a double cover $\pi'_{i}\colon C\longrightarrow \mathbb{P}^{1}$ such that $\pi'_{i}\circ \alpha_{i}=h_{i}\circ\pi_{i}$. Since $\pi_{i}^{-1}(E_{4})$ and $\pi_{i}^{-1}(E_{5})$ are rational curves they are contained in fibers of $\alpha_{i}$. So $E_{4}$ and $E_{5}$ are contained in fibers of $h_{i}$. Thus $h_{i}=u_{s_{i}}$ for some $s_{i}\in\{1,2,3\}$ by Lemma \ref{threefibers}. We obtain the following commutative diagram:
\begin{equation*}\label{isqd}
\xymatrix{
  Y_{i}  \ar[r]^{\pi_{i}}  \ar[d]_{\alpha_{i}} & S  \ar[d]^{u_{s_{i}}}\\
  C_{i} \ar[r]^{\pi'_{i}}                                 & \mathbb{P}^{1} }
\end{equation*}
By Corollary \ref{etaij} $(-\eta+\eta_{i})|_{F_{i}}\notequiv \mathcal{O}_{F_{i}}$. It implies that a general curve in $\pi^{*}_{i}(|F_{i}|)$ is connected. Hence $s_{i}\neq i$. 
\medskip

We devide the proof into six steps.
\paragraph{\textbf{Step $1\colon$} \it{The fibration $u_{i}\colon S\longrightarrow \mathbb{P}^{1},\ i=1,2,3$ has exactly two double fibers.}}
\textrm{ }

It is enough to show that $u_{3}\colon S\longrightarrow \mathbb{P}^{1}$ has at most two double fibers because $u_{3}$ already has two different double fibers, $2(E_{1}+E'_{2})$ and $2(E_{2}+E'_{1})$. Since $s_{3}\neq 3$ we may consider $u_{s_{3}}=u_{1}$.\ Assume that $u_{3}$ has one additional double fiber $2M$ aside from $2(E_{1}+E'_{2})$ and $2(E_{2}+E'_{1})$.\ Then $M$ is reduced and irreducible by Proposition \ref{finampirr} because $ME_{3}=1$ and $\varphi(M)$ is irreducible. So $\varphi$ is ramified along $M$ because the curve in the pencil $|f_{3}|$ supported on $\varphi(M)$ is reduced.

Let $R$ be the ramification divisor of the bicanonical morphism $\varphi\colon S\longrightarrow \Sigma \subset \mathbb{P}^{4}$. We have $\varphi^{*}(-K_{\Sigma})\equiv 2K_{S}$ since the image $\Sigma$ of $\varphi$ is a del Pezzo surface of degree $4$ in $\mathbb{P}^{4}$ (See Notations \ref{notdel} and \ref{BDBM1}). It implies $R\equiv K_{S}+\varphi^{*}(-K_{\Sigma})\equiv 3K_{S}$ by Hurwitz formula $K_{S}\equiv \varphi^{*}(K_{\Sigma})+R$. Put $R_{0}:=\sum_{i=1}^{3}(E_{i}+E'_{i})+G_{1}+G_{2}+H_{1}+H_{2}+M$. By Assumption \ref{ass} $\varphi$ is ramified along $E_{i},\ E'_{i},\ G_{j}$ and $H_{j}$ for $i=1,2,3$ and $j=1,2$. It follows $R_{0}\le R$. So we get a nonzero effective divisor $E:=2(R-R_{0})\equiv F_{3}-E_{4}-E_{5}$. However, it induces a contradiction because $0<EK_{S}=(F_{3}-E_{4}-E_{5})K_{S}=0$ since $K_{S}F_{3}=4$ by Proposition \ref{cong3}, $E_{4}$ and $E_{5}$ are $(-4)$-curves and $K_{S}$ is ample by Proposition \ref{finampirr}.

Similarly we get that $u_{1},\ u_{2}$ each has exactly two double fibers.
\medskip
\paragraph{\textbf{Step $2\colon$} \it{$(s_{1}\ s_{2}\ s_{3})$ is a cyclic permutation.}}
\textrm{ }

Since $s_{i}\neq i$ we need $s_{i}\neq s_{j}$ if $i\neq j$.\ We verify $s_{1}\neq s_{2}$. Otherwise, it is $s_{1}=s_{2}=3$, and $\alpha_{1}\colon Y_{1}\longrightarrow C_{1}$ ($resp.\ \alpha_{2}\colon Y_{2}\longrightarrow C_{2}$) arises in the Stein factorization of $u_{3}\circ \pi_{1}$ ($resp.\ u_{3}\circ\pi_{2}$). We have the following commutative diagram:
\begin{equation*}
\xymatrix{
  Y_{1}  \ar[r]^{\pi_{1}}  \ar[d]_{\alpha_{1}} & S  \ar[d]^{u_{3}}   &  \ar[l]_{\pi_{2}}  Y_{2} \ar[d]^{\alpha_{2}}\\
  C_{1} \ar[r]^{\pi'_{1}}                                 & \mathbb{P}^{1}   &  \ar[l]_{\pi'_{2}} C_{2}
  }
\end{equation*}
For $i=1,2$ $Y_{i}$  coincides with the normalization of the fiber product $C_{i}\times_{\mathbb{P}^{1}} S$ since $\pi_{i}$ factors through the natural projection $C_{i}\times_{\mathbb{P}^{1}} S \longrightarrow S$ which is also of degree $2$. Thus $\pi'_{1}$ is different from $\pi'_{2}$.\ We denote $q_{1},\ q_{2},\ q_{3}=u_{3}(E_{4}),\ q_{4}=u_{3}(E_{5})$ as the branch points of $\pi'_{1}$. Then we find  a branch point $q_{5}$ of $\pi'_{2}$ which is not branched over by $\pi'_{1}$. We have the fibers over the points $q_{i},\ i=1,2,5$ of $u_{3}$ are double fibers. It is a contradiction by \textbf{Step $1$}.\medskip

From now on we assume $s_{1}=2,\ s_{2}=3,\ s_{3}=1$, and for each $i\in\{1,2,3\}$ the fibration $u_{i}$ has exactly two double fibers.
\medskip
\paragraph{\textbf{Step $3\colon$} \it{$\varphi^{*}(g_{3})$ and $\varphi^{*}(h_{3})$ are not reduced.}}
\textrm{ }

We have the following commutative diagram:
\begin{equation*}
\xymatrix{
  Y_{2}  \ar[r]^{\pi_{2}}  \ar[d]_{\alpha_{2}} & S  \ar[d]^{u_{3}}\\
  C_{2} \ar[r]^{\pi'_{2}}                                 & \mathbb{P}^{1} }
\end{equation*}
Let $W$ be $C_{2}\times_{\mathbb{P}^{1}} S$, and let $p\colon W\longrightarrow S$ be the natural projection which is a double cover. Assume that $G_{3}:=\varphi^{*}(g_{3})$ ($resp.\ H_{3}:=\varphi^{*}(h_{3})$) is reduced.\ Since $\pi'_{2}\colon C_{2}\longrightarrow \mathbb{P}^{1}$ is branched over the point $u_{3}(G_{3})=u_{3}(E_{4})$ ($resp.\ u_{3}(H_{3})=u_{3}(E_{5})$), the map $p$ is branched over $G_{3}$ ($resp.\ H_{3}$).\ Thus $W$ is normal along $p^{-1}(G_{3})$ ($resp.\ p^{-1}(H_{3})$). The map $\pi_{2}\colon Y_{2}\longrightarrow S$ is also branched over $G_{3}$ ($resp.\ H_{3}$) because $Y_{2}$ is the normalization of $W$. It is a contradiction because the branch locus of $\pi_{2}$ is $E_{4}\cup E_{5}$.
\medskip
\paragraph{\textbf{Step $4\colon$}\textit{A general element $F_{i}\in|F_{i}|$ is hyperelliptic for each $i\in\{1,2,3\}$.}}
\textrm{ }

We verify that a general fiber $F_{2}\in|F_{2}|$ is hyperelliptic.\ Since the pull-back $\pi_{1}^{*}(F_{2})$ ($resp.\ \pi_{1}^{*}(F_{3})$) is disconnected, we may consider $\pi_{1}^{*}(F_{2})=\hat{F_{2}}+{\hat{F_{2}}}'$ ($resp.\ \pi_{1}^{*}(F_{3})=\hat{F_{3}}+{\hat{F_{3}}}'$) where the two components are disjoint.\ Then we get $\hat{F_{2}}\hat{F_{3}}=2$ by $F_{2}F_{3}=4$. Let $p\circ h\colon Y_{1}\longrightarrow C\longrightarrow \mathbb{P}^{1}$ be the Stein factorization of $u_{3}\circ \pi_{1}\colon Y_{1}\longrightarrow \mathbb{P}^{1}$. Since $\hat{F_{3}}$ is a fiber of $h\colon Y_{1}\longrightarrow C$, the restriction map $h|_{\hat{F_{2}}}\colon \hat{F_{2}}\longrightarrow C$ is a $2$-to-$1$ map by $\hat{F_{2}}\hat{F_{3}}=2$. Moreover, $C$ is rational because $h\colon Y_{1}\longrightarrow C$ is not Albanese map and $q(Y_{1})=1$. Thus $\hat{F_{2}}$ is hyperelliptic, and so is $F_{2}$.

\medskip
\paragraph{\textbf{Step $5\colon$}\textit{$\varphi\colon S\longrightarrow \Sigma$ is a Galois cover with the Galois group $G\cong\mathbb{Z}_{2}\times\mathbb{Z}_{2}$.}}
\textrm{ }

For each $i\in\{1,2,3\}$, let $\gamma_{i}$ be the involution on $S$ induced by the involution on the general fiber $F_{i}$. Since $S$ is minimal, the maps $\gamma_{i}$ are regular maps. So the maps $\gamma_{i}$ belong to $G$ by Proposition \ref{FijKi}. Now it suffices $\gamma_{i}\neq \gamma_{j}$ if $i\neq j$. We show $\gamma_{2}\neq \gamma_{3}$. Consider the lifted involution $\hat{\gamma_{2}}\colon Y_{1}\longrightarrow Y_{1}$. The restriction of $\alpha_{1}$ identifies $\hat{F_{3}}/ \hat{\gamma_{2}}$ with $C_{1}$ by the construction in Step $4$. Thus we obtain $p_{g}(\hat{F_{3}}/\hat{\gamma_{2}})=1$, but $\hat{F_{3}}/ \hat{\gamma_{3}}\cong \mathbb{P}^{1}$. It means that $\gamma_{2}\neq \gamma_{3}$.
\medskip

\paragraph{\textbf{Step $6\colon$}\ \textit{$S$ is a Burniat surface.}}
\textrm{ }

Denote by $B$ be the branch divisor of $\varphi$. Then we get
\[-3K_{\Sigma}\equiv B \ge \sum_{i=1}^{3}(e_{i}+e'_{i}+g_{i}+h_{i})\equiv -3K_{\Sigma},\]
thus $B=\sum_{i=1}^{3}(e_{i}+e'_{i}+g_{i}+h_{i})$. Denote $B_{i}$ as the image of the divisorial part of the fixed locus of $\gamma_{i}$. We have $B=B_{1}+B_{2}+B_{3}$. By \textbf{Step $4$} we obtain $B_{i}=e_{i}+e'_{i}+g_{i+1}+h_{i+1}$ for each $i\in\{1,2,3\}$, where $g_{4}$ ($resp.\ h_{4}$) denotes $g_{1}$ ($resp.\  h_{1}$).

 The theorem is proved with all steps. \hfill $\Box$ \par

\medskip

{\em Acknowledgements}. This work was supported by Shanghai Center for Mathematical Sciences. The author is very grateful to the referee for valuable suggestions and comments. 

\bigskip

\begin{small}

\end{small}

\end{document}